\pdfoutput=1
\documentclass[oneside]{article}

\usepackage{geometry}
\usepackage{mathtools}
\usepackage{amssymb}
\usepackage{amsthm}
\usepackage{stix}
\usepackage{tikz-cd}
\usepackage{float}
\usepackage{hyperref}
\usepackage[noabbrev,nameinlink,capitalise]{cleveref}

\usepackage{afterpage}
\usepackage{pdflscape}
\usepackage{spectralsequences}

\hypersetup{
    colorlinks,
    linkcolor={red!30!black},
    citecolor={blue!50!black},
    urlcolor={blue!80!black}
}

\usepackage{etoolbox}
\makeatletter
\patchcmd\@ynthm{%
    \cref@stack@add{#1}{\cref@label@types}%
}{%
    \edef\temp{\if@cref@capitalise\noexpand\MakeUppercase\else\noexpand\MakeLowercase\fi}%
    \expandafter\def\expandafter\temp\expandafter{\temp #3}
    \edef\temp{\@nx\crefname{#1}{\unexpanded\@xp{\temp}}{\unexpanded\@xp{\temp}s}}
    \temp
    \Crefname{#1}{#3}{#3s}%
}{}{\failed}

\def\my@cref@label#1{\@xp\@xp\@xp\my@@cref@getlabel\csname r@#1@cref\endcsname{}{}}
\def\my@@cref@getlabel#1#2{\my@@cref@getlabel@#1&[][][]&\@nil}
\def\my@@cref@getlabel@#1][#2][#3]#4&#5\@nil{#4}
\def\my@cref@type#1{\@xp\@xp\@xp\my@@cref@gettype\csname r@#1@cref\endcsname{}}
\def\my@@cref@gettype#1#2{\my@@cref@gettype@#1&[][][]!&\@nil}
\def\my@@cref@gettype@#1][#2][#3]#4&#5\@nil{\ifx!#4\else\@gobble#1\@xp\@firstofone\fi}

\makeatother

\newtheorem{thm}{Theorem}[section]
\newtheorem{lem}[thm]{Lemma}
\newtheorem{prop}[thm]{Proposition}

\theoremstyle{definition}

\let\susp\Sigma
\let\loops\Omega
\def\suspinfty{\Sigma^{\infty}}
\def\Z{\mathbb{Z}}
\let\sm\wedge

\def\tmftw{\mathrm{tmf}_{(2)}}
\def\Ftw{\mathbb{F}_2}
\def\Fbar{\overline{F}}
\def\xbar{\overline{x}}
\def\nubar{\overline{\nu}}
\DeclareMathOperator\Ext{Ext}

\def\BGL{\mathop{\mathrm{BGL}_1}}
\def\BGLS{\BGL(S)}

\makeatletter
\def\acksec{%
  \@startsection
    {section}{1}{\z@}{18\p@ \@plus 2\p@ \@minus 2\p@}%
    {6\p@}{\normalfont\large\bf}}

\makeatother

\title{Thom Complexes and the Spectrum $\tmftw$}
\author{Hood Chatham}

\begin{document}
\maketitle
\begin{abstract}
Many interesting spectra can be constructed as Thom spectra of easily constructed bundles. Mahowald~\cite{Mahowald} showed that $bu$ and $bo$ cannot be realized as $E_1$ Thom spectra. We use related techniques to show that $\tmftw$ also cannot be realized as an $E_1$ Thom spectrum.
\end{abstract}

\section{Introduction}
\begin{thm}
\label{main}
The spectrum $\tmftw$ is not a Thom spectrum of an $H$-map from a loop space to $\BGLS$.
\end{thm}
We will prove this by contradiction. 
\begin{figure}
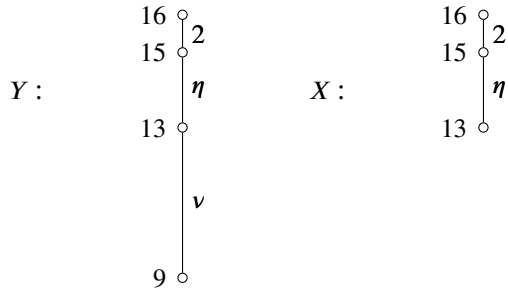

\centering
\begin{sseqpage}[no axes, yscale=0.5, class labels={left=0.2em},xscale=2, x range={-1}{3}]
\class["9"](0,8)
\class["13"](0,12)
\structline["\nu"]
\class["15"](0,14)
\structline["\eta"]
\class["16"](0,15)
\structline["2"]
\node at (-1,13) {Y\colon};
\node at (1,13) {X\colon};
\class["13"](2,12)
\class["15"](2,14)
\structline["\eta"]
\class["16"](2,15)
\structline["2"]
\end{sseqpage}
\label{fig:X-Y-cell-structures}
\caption{The cell structures of spaces $X$ and $Y$.}
\end{figure}
We show:
\begin{prop}
\label{thom-implies-perfect-cup-pairing}
Suppose that $Z$ is a loop space and $f\colon Z\to \BGLS$ is an $H$-map such that the Thom spectrum of $f$ is $\tmftw$. Then there are spaces $X$ and $Y$ with cell structures as in \cref{fig:X-Y-cell-structures} and a map $g\colon \susp^8X\to Y$ such that the cohomology of the cofiber $C$ has a cup product $x_9x_{13}=x_{22}$ where $x_9\in H^9(C;\Z)$, $x_{13}\in H^{13}(C;\Z)$, and $x_{22}\in H^{22}(C;\Z)$ are generators.
\end{prop}
%
\begin{prop}
\label{extract-three-cell}
Suppose $X$ and $Y$ are spaces with cell structures as in Figure~\ref{fig:X-Y-cell-structures}, suppose $g\colon \susp^8X\to Y$ is any map and $C$ is the cofiber of $g$. Then there is a space $D$ with $H^*(D;\Z)\cong \Z\{x_{9},x_{13},x_{22}\}$ and a map $D\to C$ inducing a surjection on cohomology.
\end{prop}
If we take the map $g$ to be the map produced in \cref{thom-implies-perfect-cup-pairing}, then the resulting space $D$ has a perfect cup pairing $x_{9}x_{13}=x_{22}$. James has a classification theorem that says what attaching maps and cup product structures are possible on 3-cell CW complexes~\cite[Theorem 1.2]{james}. Using this we show:
\begin{prop}
\label{no-perfect-cup-pairing}
Suppose $D$ is a space with $H^*(D;\Z)\cong \Z\{x_{9},x_{13},x_{22}\}$ and $Sq^4(\xbar_{9})=\xbar_{13}$ where $\xbar_i$ denotes the image of $x_i$ under the reduction map $H^*(D;\Z)\to H^*(D;\Z/2)$. Then $x_9x_{13}=2kx_{22}$ for some $k\in \Z$.
\end{prop}
\Cref{extract-three-cell,no-perfect-cup-pairing} show that the conclusion of \cref{thom-implies-perfect-cup-pairing} is a contradiction, which proves \cref{main}.

\subsection{Comparison to Mahowald}
Our argument is closely based on Mahowald's argument in~\cite{Mahowald} that $bo$ is not a Thom spectrum.  For comparison, we reformulate Mahowald's argument in a parallel form to ours to make the similarities and the differences apparent.  
\begin{figure}
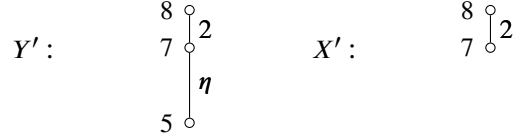

\centering
\begin{sseqpage}[no axes, yscale=0.5, class labels={left=0.2em},xscale=2, x range={-1}{3}]
\class["5"](0,5)
\class["7"](0,7)
\structline["\eta"]
\class["8"](0,8)
\structline["2"]
\node at (-1,7) {Y'\colon};
\node at (1,7) {X'\colon};
\class["7"](2,7)
\class["8"](2,8)
\structline["2"]
\end{sseqpage}
\label{fig:X'-Y'-cell-structures}
\caption{The cell structures of $X'$ and $Y'$.}
\end{figure}
\Cref{thom-implies-perfect-cup-pairing} is an analogue of:
\begin{prop}[{\cite[Discussion on page 294]{Mahowald}}]
\label{Mah-thom-implies-perfect-cup-pairing}
Suppose that $Z$ is a loop space and $f\colon Z\to \BGLS$ is an $H$-map such that the Thom spectrum of $f$ is $bo$. Then there are spaces $X'$ and $Y'$ with cell structures as indicated in \cref{fig:X'-Y'-cell-structures} and a map $\susp^4 X'\to Y'$ such that the cohomology of the cofiber $C'$ has a cup product $x_5x_{7}= x_{13}$ where $x_5\in H^5(C';\Z)$, $x_{7}\in H^{7}(C';\Z)$, and $x_{12}\in H^{12}(C';\Z)$ are generators.
\end{prop}
\Cref{extract-three-cell} is an analogue of:
\begin{prop}[{\cite[Lemma 3 and discussion on page 294]{Mahowald}}]
\label{Mah-extract-three-cell}
Suppose $g\colon \susp^4X'\to Y'$ is any map and $C'$ is the cofiber. Then there is a space $D'$ with $H^*(D';\Z)\cong \Z\{x_{5},x_{7},x_{12}\}$ and a map $D'\to C'$ inducing a surjection on cohomology.
\end{prop}
Again taking $g$ to be the map produced in \cref{thom-implies-perfect-cup-pairing}, then the resulting space $D'$ has a perfect cup pairing $x_{5}x_{7}=x_{12}$. This means that $D'$ is an $S^5$ bundle over $S^7$. The 7-cell in $D'$ is attached to the 5-cell by an $\eta$ so $D'$ has no section. Mahowald deduces a contradiction:
\begin{lem}[{\cite[Lemma 4]{Mahowald}}]
\label{sphere-bundles-have-section}
Every 5-sphere bundle over $S^7$ has a section.
\end{lem}

The proof of \cref{thom-implies-perfect-cup-pairing} is exactly the same as the proof of \cref{Mah-thom-implies-perfect-cup-pairing}, we merely fill in details. The proof of \cref{extract-three-cell} is completely different from the proof of \cref{Mah-extract-three-cell}. The analog of \cref{sphere-bundles-have-section} in our setting would state that every $9$-sphere bundle over $S^{13}$ has a section, but this is false -- using
\cite[Theorem 1.2]{james}, it is possible to show that there exists a space $D$ with $H^*(D;\Z)=\Z\{x_{9},x_{13},x_{22}\}$, with $D^{(13)} \simeq C(2\nu_{9})$ and with $x_{9}x_{13}=x_{22}$. This is an $S^9$ bundle over $S^{13}$ with no section. We deduce \cref{sphere-bundles-have-section} from the following analog of \cref{no-perfect-cup-pairing}:
\begin{prop}
\label{Mah-no-perfect-cup-pairing}
Suppose $D'$ is a space with $H^*(D';\Z)\cong \Z\{x_{5},x_{7},x_{12}\}$ and $Sq^2(\xbar_{5})=\xbar_{7}$ where $\xbar_i$ denotes the image of $x_i$ under the reduction map $H^*(D';\Z)\to H^*(D';\Z/2)$. Then $x_9x_{13}=2kx_{22}$ for some $k\in \Z$.
\end{prop}

\begin{proof}[Proof that \cref{Mah-no-perfect-cup-pairing} implies \cref{sphere-bundles-have-section}]
A 5-sphere bundle over $S^7$ is a space $D'$ with a 5-cell, a 7-cell, and a 12-cell, where in $H^*(X)$, $x_5x_7=x_{12}$. Such a space has a section if the attaching map $S^6\to S^5$ of the 7-cell is null. \Cref{Mah-no-perfect-cup-pairing} says that the attaching map cannot be $\eta_5$, so the remaining possibility is that it is null.
\end{proof}
Mahowald has a different proof of \cref{sphere-bundles-have-section}.

\subsection*{Acknowledgements}
Thanks to Sanath Devalapurkar for suggesting this project and to Robert Burklund and Andy Senger for helpful discussions. Thanks to Neil Strickland for his marvelous \href{http://neil-strickland.staff.shef.ac.uk/toda/}{unstable homotopy theory Mathematica package} which was of great help in writing section 4.

\section{Proof of \Cref{thom-implies-perfect-cup-pairing}}
For a space $X$ let $X^{(d)}$ denote the $d$-skeleton of $X$ and let $X_{(d)}$ denote the cofiber of the inclusion map $X^{(d-1)}\to X$.

\begin{proof}[Proof of \cref{thom-implies-perfect-cup-pairing}]
Suppose that $Z$ is a loop space and $Z\to \BGLS$ is an $H$-map such that the Thom spectrum is $\tmftw$. Then $H_*(Z;\Ftw)\cong \Ftw[x_{8},x_{12},x_{14},x_{15},x_{31},\ldots]$ where under the Thom isomorphism $H_*(Z)\cong H_*(\tmftw;\Ftw)=\Ftw[\xi_1^8,\xi_2^4,\xi_3^2,\xi_4,\xi_5,\ldots]$, the class $x_i$ maps to the multiplicative generator of $H_*(\tmftw;\Ftw)$ in the same degree. Because $Z$ is a loop space, the inclusion $Z^{(15)}\to Z$ of the 15 skeleton of $Z$ extends to a map $f\colon\loops\susp Z^{(15)}\to Z$. Let $F$ be the fiber of $f$ and let $\Fbar = F^{(25)}$. Let $g\colon \susp \Fbar \to \susp Z^{(15)}$ be the adjoint to inclusion of fiber map $\Fbar\to Z^{(15)}$. The Steenrod action on the homology of $\susp Z^{15}$ shows that it has the cell structure indicated for the space $Y$, so we can take $Y=\susp Z^{15}$ and $X=Y_{(13)}$. By \cref{Fbar=suspX,cohomology-of-C}, we are finished.
\end{proof}

\begin{lem}
\label{Fbar=suspX}
There is a unique space $X$ with cell structure as in \cref{fig:X-Y-cell-structures}. In the context of the proof of \cref{thom-implies-perfect-cup-pairing}, $\susp\Fbar \simeq \susp^{8}X$.
\end{lem}

\begin{proof}
Suppose $X_1$ and $X_2$ are two spaces with the cell structure as in \cref{fig:X-Y-cell-structures}. Since the bottom cell of $X$ is in dimension 13 and the top cell is in dimension $16<2\times 13 - 1$, there is an isomorphism $[X_1, X_2]\to [\susp^{\infty +}X_1, \susp^{\infty}X_2]$ so it suffices to show that $\susp^\infty X$ is uniquely determined by its $\Ftw$ cohomology. The $E_2$ page of the Adams spectral sequence $\Ext(H_*(X_1;\Ftw),H_*(X_2;\Ftw))$ is displayed in \cref{ASS}, and the ($-1$)-stem is empty, so $X$ is uniquely determined by its cohomology.

So to show $\susp\Fbar \simeq \susp^{8}X$ it suffices to check that that $H^*(\susp \Fbar;\Ftw)\cong H^*(\susp^{8}X;\Ftw)$. We are computing the fiber of $f\colon \loops\susp Z^{(15)}\to Z$ through dimension 23. The homology $H_{*}(\loops\susp Z^{(15)};\Ftw)$ is the free associative algebra on $H_*(Z^{(15)};\Ftw)$ and $H_{*}(f;\Ftw)$ is the map $\Ftw\langle x_{8},x_{12},x_{14},x_{15}\rangle \to \Ftw[x_8,x_{12},x_{14},x_{15},x_{31},\ldots]$ from the associative algebra to the commutative algebra with kernel generated by commutators, surjective through dimension $30$. Thus, $f$ is $20$-connective and since $Z$ is $8$-connective, and the sequence $0\to H_*(F)\to H_*(\loops\susp Z^{(15)})\to H_*(Z)\to 0$ is exact through dimension 27. Thus, $H_{*}(\Fbar)=\Ftw\{[x_8,x_{12}],[x_{8},x_{14}], [x_{8},x_{15}]\}$, where the coaction comes from the coaction on the $x_{i}$'s. We conclude that $H_{*}(\susp\Fbar;\Ftw)\cong H_{*}(\susp^8X;\Ftw)$.
\end{proof}

\begin{lem}
\label{cohomology-of-C}
In the context of the proof of \cref{thom-implies-perfect-cup-pairing}, the space $C$ has cohomology ring as follows:
$H^*(C;\Ftw) = \Ftw\{\alpha_{9}, \beta_{13},\gamma_{15},\delta_{16},\alpha_9\beta_{13},\alpha_{9}\gamma_{15},\alpha_{9}\gamma_{16}\}$. The Steenrod action is generated by $Sq^4(\alpha_{9})=\beta_{13}$, $Sq^{2}(\beta_{13})=\gamma_{15}$ and $Sq^1(\gamma_{15})=\delta_{16}$.
\end{lem}

\begin{proof}
The fiber sequence $F\to \loops\susp Z^{(15)}\to Z$ deloops to $F'\to \susp Z^{(15)}\to BZ$, so $F=\loops F'$.
Let $h\colon \Fbar\to \loops\susp Z^{(15)}$ be the inclusion of the fiber. The map $g\colon \susp\Fbar \to \susp Z^{(15)}$ is the composite $\susp\Fbar\to \Fbar' \to \susp Z^{(15)}$. Since the composition $\Fbar'\to \susp Z^{(15)}\to BZ$ is null there are maps in the following diagram:
\[\begin{tikzcd}
\mathllap{\susp \Fbar\simeq{}}\susp \loops \Fbar' \dar\drar & & C\dar\\
F' \rar & \susp X\urar\rar\drar& C'\dar\\
        && BZ \\
\end{tikzcd}\]
where $C'$ is the cofiber of $F'\to \susp X$. I claim that the maps $C\to C'$ and $C'\to BZ$ induce isomorphisms in cohomology in degree $\leq 22$. Since $BZ$ is $9$-connective and $F'$ is $21$-connective, by \cite[Theorem 6.1]{whitehead} the map $C'\to BZ$ induces an isomorphism in cohomology through degree $21+9-1=29$ and the map $\susp\loops F'\to F'$ induces an isomorphism in cohomology through degree $21 + 20 - 1  = 40$ so the map $C\to C'$ induces an isomorphism in cohomology through dimension $41$. It remains to compute the cohomology of $BZ$ in this range. The homology of $Z$ is polynomial, so the cohomology is a divided power algebra $H^*(Z)=\Gamma[y_{8},y_{12},y_{14},y_{15},y_{31},\ldots]$ so it follows that $H^*(BZ)=\Lambda(\sigma y_8,\sigma y_{12}, \sigma y_{14},\sigma y_{15}, \sigma y_{31},\ldots)$.
\end{proof}

\section{Extracting the three-cell complex}

\afterpage{
\begin{landscape}
\pagestyle{empty}
\newgeometry{margin=1cm,top=7.5cm}
\begin{figure}[p]%
    \sseqset{classes=fill}
    \SseqErrorToWarning{range-overflow}
    \input{ass-F-eta2-eta2}
    \input{ass-F-eta2-nueta2}
    \begin{sseqpage}[name=F-eta2-nueta2,xscale=1.7, yscale=0.7, y range={0}{10},x range={-3}{11}, title={$\Ext^{**}(H^*X,H^*Y)$}]
    \classoptions[red](8,1)
    \classoptions[red](8,2,1)
    \classoptions[red](8,2,2)
    \classoptions[red](11,2)
    \classoptions[red](11,3)
    \end{sseqpage}
    \par
    \begin{sseqpage}[name=F-eta2-eta2,xscale=1.7, yscale=0.7, y range={0}{7},x range={-3}{11}, title={$\Ext^{**}(H^*X,H^*X)$}]
    \classoptions["\eta\sigma" {left,pin}](8,2,1)
    \classoptions["\epsilon" {left,pin}](8,3)
    \classoptions[red,"c" {below,pin, black}](8,2,2)
    \classoptions[red,"\nu c"{right=0.2em, black}](11,3)
    \structlineoptions[red](8,2,2)(11,3)
    \end{sseqpage}
    \caption{The $E_2$ pages of the Adams spectral sequences computing $\pi_*F(\suspinfty X,\suspinfty Y)$ (top) and $\pi_*F(\suspinfty X,\suspinfty X)$ (bottom).}
    \label{ASS}
\end{figure}
\restoregeometry
\end{landscape}
}

In this section we prove \cref{extract-three-cell}. We show in \cref{pi8-extensions-are-products} that any composite $\susp^8X\to Y \to X$ is a smash product $\alpha\sm id_{X}$ for some $\alpha\in \pi_8(S)$. Let $i\colon S^{21}\to \susp^8X$ be the inclusion of the bottom cell. In \cref{extract-three-cell-main} we show that because the map $\susp^8X\to Y \to X$ is a smash product, any composite $S^{21}\to \susp^8X\to Y\to Y_{(15)}$ is null, We deduce that $g\circ i$ factors through the $21$ skeleton of the fiber of $Y\to Y_{(15)}$, which is $C(\nu_9)$ by \cref{fiber-of-coskeleton}. From this we deduce \cref{extract-three-cell}.

\begin{lem}
The map $\pi_8(S) \to [\susp^8X,X]$ given by $\alpha\mapsto \alpha\sm id_{X}$ is injective.
\label{pi8-injective}
\end{lem}

\begin{proof}
First note that since $X$ has its bottom cell in dimension $13$ and $\susp^8X$ has its top cell in dimension $23$ which is less than or equal to $2\times 13-2$, there is an isomorphism $[\susp^8X,X] \to [\susp^{\infty+8}X,\susp^{\infty}X]$. I claim that the further composition $\pi_8(S)\to  [\susp^{\infty+8}X,\susp^{\infty}X] \to [\susp^{\infty+8}X,\susp^{\infty}S^{16}]$ is injective, where the second map is squeezing off to the top cell of $X$. Let $A=\susp^{13}D\susp^{\infty}X$, which has the following cell structure:
\newenvironment{nscenter}
 {\parskip=0pt\par\nopagebreak\centering}
 {\par\noindent\ignorespacesafterend}
\begin{center}
\begin{sseqpage}[no axes, yscale=0.5, class labels={left=0.2em}]
\makeatletter
\begin{scope}[background]
\end{scope}
\class["0"](0,0)
\class["1"](0,1)
\structline["2"]
\class["3"](0,3)
\structline[bend right=20,"\eta"]
\end{sseqpage}
\end{center}
The map $\pi_8(S)\to  [\susp^{\infty+8}X,\susp^{\infty}X] \to \pi_8(A)$ is induced by the inclusion of the bottom cell of $A$. Thus, we need to show that no Atiyah Hirzebruch differentials hit $\eta\sigma$ or $\epsilon$ on the bottom cell of $A$. The first differential is multiplication by two, and since $\pi_8(S)$ is all two-torsion, this does not hit anything. The second differential is given by the Toda bracket $\left<-,\eta,2\right> : \ker(\eta\colon \pi_{6}(S)\to \pi_7(S)) \to \pi_{8}(S)$. Since $\pi_6(S)$ consists just of $\nu^2$, it suffices to show that the Toda bracket $\left<\nu^2,\eta,2\right>=0$. By shuffling, $\left<\nu^2,\eta,2\right>=\left<\nu,\nu\eta,2\right>=0$ because $\nu\eta=0$, and the indeterminacy is the image of $\nu^2$ and $2$ in $\pi_8$ which is trivial.
\end{proof}

\begin{lem}
The image of $\pi_{8}F(X,Y) \to \pi_8F(X,X)$ is contained in the image of the map $\pi_{8}S \to \pi_8F(X,X)$ -- that is, they are maps of the form $\alpha\sm id_X$.
\label{pi8-extensions-are-products}
\end{lem}

\begin{proof}
Refer to \cref{ASS}. By \cref{pi8-injective}, the two classes labeled $\eta\sigma$ and $\epsilon$ are the images of $\eta\sigma,\epsilon \in \pi_8(S)$. Note that the 8-stem in $\Ext^{**}(H^*X,H^*X)$ is $\Z/2\{\epsilon,\eta\sigma, c\}$ so that $\pi_8F(X,X)$ is either $(\Z/2)^2$ or $(\Z/2)^3$ depending on whether or not the class $c$ supports an Adams $d_2$ hitting $8\sigma$. If it does support such a differential, the map $\pi_8S\to \pi_8F(X,X)$ is surjective and we're done. Otherwise, it suffices to show that $c$ is not in the image of the map $p\colon \pi_8(\suspinfty X,\suspinfty Y)\to \pi_8(\suspinfty X, \suspinfty X)$. I claim that for all $x$ in the $8$-stem of $\Ext^{**}(H^*X,H^*Y)$, $\nu x$ is detected in filtration at least $4$. Since $\nu c$ is nonzero and in bidegree $(11,3)$, this implies that $c$ is not in the image of $p$. To see that $\nu x$ is detected in filtration at least $4$ note that $\nu$ multiplication on the $8$-stem is zero in the associated graded, so multiplication by $\nu$ raises filtration by at least 2. This implies that $\nu x$ cannot be the class in $(11,2)$. Since the class in $(11,3)$ is $256$-torsion, it can't be divisible by $\nu$ so $\nu x$ is detected in filtration at least $4$ as needed.
\end{proof}

\afterpage{
    \begin{sseqdata}[ no ticks, scale=0.5,yscale=0.3,cohomological Serre grading, name=SSS-Y-to-Y15]
\begin{scope}[ background ]
\foreach \x in {0,15,16}{
    \node at (\x,\ymin - 6) {\x};
}

\foreach \y in {0,9,13,23}{
    \node at (\xmin-2,\y) {\y};
}
\end{scope}
\foreach \y in {0,9,13,23}{
    \class(0,\y+0)
    \class(15,\y+0)
    \class(16,\y+0)
    \structline
}
\structline["\nu"](0,9)(0,13)
\d15(0,23)
\end{sseqdata}

    \begin{figure}
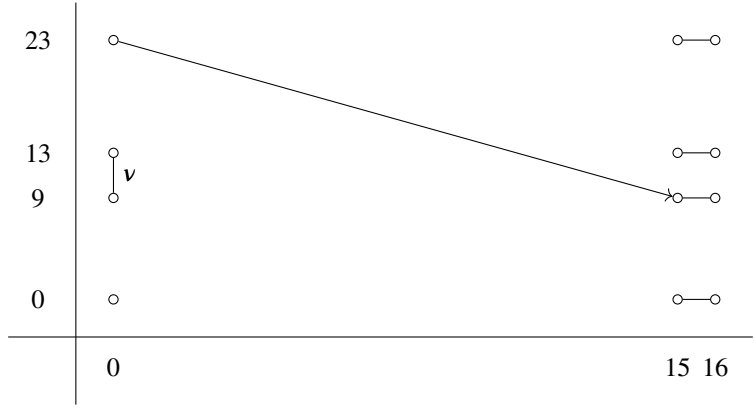

    \centering
    \printpage[name=SSS-Y-to-Y15]
    \caption{The Serre spectral sequence for $F\to Y\to Y_{(15)}$}
    \label{SSS-Y-to-Y15}
    \end{figure}
}

\begin{lem}
\label{fiber-of-coskeleton}
Let $F$ be the fiber of $Y\to Y_{(15)}$. Then $F^{(21)}\simeq C(\nu_9)$.
\end{lem}

\begin{proof}
Since the composite $C(\nu_9)=Y^{(13)}\to Y\to Y_{(15)}$ is null, there is a natural map $C(\nu_9)\to F$. We compute the Serre spectral sequence for the cohomology of the fiber sequence $F\to Y\to Y_{(15)}$ and see that the map $C(\nu_9)\to F$ is an equivalence through degree 22. See \cref{SSS-Y-to-Y15}.
\end{proof}

\begin{lem}
\label{extract-three-cell-main}
There is a commutative square:
\[\begin{tikzcd}
S^{21}\dar \rar & C(\nu_{9}) \dar \rar & D\dar\\
\susp^8 X\rar["g"]& Y \rar               & C\\
\end{tikzcd}\]
where rows are cofiber sequences, the map $S^{21}\to \susp^8X$ is the inclusion of the bottom cell and the map $C(\nu_{9})\to Y$ is the inclusion of the fiber of $Y\to Y_{(15)}$. The map $D\to C$ is an isomorphism in cohomology in degrees $9$, $13$, and $22$.
\end{lem}

\begin{proof}
By \cref{pi8-extensions-are-products}, the composite $\susp^8 X\to Y \to X$ is a smash product $\alpha\sm id_{X}$ for some $\alpha\in \pi_8 S$.
We get a commutative square:
\[\begin{tikzcd}
S^{21} \ar[rr,"\alpha"]\dar["\susp^8i"] & &S^{13}\rar\dar["i"] & *\dar\\
\susp^8 X \rar["g"]\ar[rr,bend right=20,"\alpha\sm id_X"',yshift=-2pt] & Y \rar & X \rar & X_{(15)}
\end{tikzcd}\]
thus the composite of $g\circ \susp^8 i \colon S^{21}\to Y$ with the projection $Y\to Y_{(15)}$ is null and factors through the fiber $F$ of $Y\to Y_{(15)}=X_{(15)}$. In fact it factors through $F^{(21)}\simeq C(\nu_9)$. Thus there is a map $h$ making the following diagram commute:
\[\begin{tikzcd}
S^{21}\dar \rar["h"] & C(\nu_9)\dar \\
\susp^8 X\rar["g"] & Y \\
\end{tikzcd}\]
The map $S^{21}\to \susp^8X$ is an isomorphism in cohomology through dimension 21, and the map $C(\nu_9)\to Y$ is an isomorphism in cohomology through dimension 13 and also in dimension 22, so the map $D\to C$ is an isomorphism in cohomology in dimensions $9$, $13$ and $21$.
\end{proof}
\Cref{extract-three-cell} is an immediate consequence of \cref{extract-three-cell-main}.

\section{Unstable calculations to prove \Cref{no-perfect-cup-pairing}}
The main ingredient of \cref{no-perfect-cup-pairing} is the following theorem of James, which tells us which cup product structures on 3-cell complexes exist. Suppose that $K$ is a three cell CW complex with cells in dimension $q$, $n$, and $n+q$. For $\alpha\in \pi_{n-1}S^q$ and $m$ an integer, say that $K$ has type $(m,\alpha)$ if the attaching map of the $n$ cell to the $q$ cell is given by $\alpha$ and the integral cohomology $H^*(K;\Z) = \Z\{x_{q},y_{n},z_{n+q}\}$ has cup product $x_qy_n=mz_{n+q}$.

\begin{thm}[{\cite[Theorem 1.2]{james}}]
\label{jamesthm}
Let $\alpha\in \pi_{n-1}(S^q)$ where $n-1> q\geq 2$. Let $[\alpha,i_q]$ denote the Whitehead product of $\alpha$ and a generator $i_q\in \pi_q(S^q)$. There exists a complex $K$ of type $(m,\alpha)$ if and only if $m[\alpha,i_q]$ is contained in the image of left composition with $\alpha$: $\alpha_*\colon \pi_{n+q-2}(S^{n-1})\to \pi_{n+q-2}(S^q)$.
\end{thm}
We apply this with $\alpha=\nu_9$ to show that no three-cell complex of type $(1,\nu_9)$ exists, which is a reformulation of \cref{no-perfect-cup-pairing}:
\begin{prop}
\label{no-three-cell}
A three cell complex $D$ of type $(m,\nu_9)$ exists if and only if $m$ is even.
\end{prop}

\begin{lem}
\label{nu9-mult-zero}
The map $\nu_9\colon \pi_{20}(S^{12})\to \pi_{20}(S^9)$ is zero.
\end{lem}

\begin{proof}
According to \cite[Theorem 7.1]{toda}, $\pi_{20}S^{12}=\Z/2\{\epsilon_{12},\nubar_{12}\}$. By \cite[Equations 7.18]{toda}, $\nu_6\circ \epsilon_{9}=2\nubar_6\circ \nu_{14}$ so suspending this gives $\nu_{9}\circ \epsilon_{12}=2\nubar_9\circ \nu_{17}$. According to \cite[Theorem 7.1]{toda}, $2\nubar_7=0$ so $\nu_{9}\circ \epsilon_{12}=0$.

By \cite[Lemma 6.4]{toda}, $\eta_{12}\circ \sigma_{13}=\epsilon_{12}+\nubar_{12}$. Since $\nu_6\circ \eta_{9}=0$, we see that $\nu_9\circ \nubar_{12} = \nu_9\circ (\epsilon_{12}+\nubar_{12}) = \nu_9\circ \eta_{12}\circ \sigma_{13} = 0$.
\end{proof}

\begin{lem}
\label{whitehead-product}
The Whitehead product $[\nu_9,i_{9}]$ is of order two.
\end{lem}

\begin{proof}
First note that $[\nu_9,i_9]= [i_9,i_9]\circ \nu_{17}$. By \cite[Theorem 7.1]{toda}, $\pi_{17}S^{9}=\Z/2\{\epsilon_9,\nubar_9,\sigma_9\circ\eta_{16}\}$, where under suspension $\nubar_{10}=\epsilon_{10} + \sigma_{10}\circ\eta_{17}$. By \cite[Theorem 7.4]{toda}, $\pi_{20}S^9 = \Z/8\{\zeta_9\} \oplus \Z/2\{\nubar_9\circ\nu_{17}\}$, where the element $\zeta_9$ is stably $P(\nu)$ and $\nubar_{9}\circ\nu_{17}$ is in the kernel of suspension. According to \cite[Equation 7.1]{toda}, $[i_9,i_9]=\nubar_{9}+\epsilon_{9} + \sigma_9\circ\eta_{16}$ is the nonzero element of the kernel of suspension in $\pi_{17}S^9$. Since $\eta_{16}\circ \nu_{17}= 0 = \epsilon_{9}\circ \nu_{17}$, we have $[\nu_9,i_9] = [i_9,i_9]\circ \nu_{17} = \nubar_{9}\circ\nu_{17}$ is the nontrivial element of the kernel of suspension in $\pi_{20}S^9$.
\end{proof}

\begin{proof}[Proof of \cref{no-three-cell}]
We apply \cref{jamesthm} with $q=9$, $n=13$, $\alpha=\nu_9$. By \cref{whitehead-product}, the Whitehead product $[\nu_9,i_9]$ is a nonzero two-torsion element, but by \cref{nu9-mult-zero}, the image of $\nu_9\colon \pi_{20}S^{12}\to \pi_{20}S^9$ is zero.  A type $(m,\nu_9)$ complex exists if and only if $m[\nu_9,i_9]=0$ which is true when $m$ is even.
\end{proof}

\bibliography{tmf-not-E1-thom-spectrum}
\bibliographystyle{plain}

\end{document}